\documentclass[12pt]{article}

\usepackage{fullpage}
\usepackage{amssymb}
\usepackage{amsmath}
\usepackage{amsthm}
\usepackage{url}
\usepackage{hyperref}
\usepackage{graphicx}
\usepackage{mathabx}
\usepackage{tablefootnote}

\newtheorem{theorem}{Theorem}
\newtheorem{corollary}{Corollary}

\newtheorem{lemma}{Lemma}

\newtheorem{definition}{Definition}
\DeclareMathOperator{\sign}{sign}

\newcommand{\R}{\mathbb{R}}
\newcommand{\cG}{{\cal{G}}}

\newcommand{\ijktuple}{(I, +_I, \preceq_I, J, +_J, \preceq_J, K, +_K, \preceq_K, \ast)}
\newcommand{\ituple}{(I, +, \preceq_I, \ast)}

\begin{document}

\title{On generalizations of Schur's inequality}

\date{August 31, 2020\\Latest update: July 7, 2021}
\author{Chai Wah Wu\\IBM T. J. Watson Research Center\\P. O. Box 218, Yorktown Heights, New York 10598, USA\\e-mail: chaiwahwu@ieee.org}
\maketitle

\begin{abstract}
Schur's inequality for the sum of products of the differences of real numbers states that for $x,y,z,t\geq 0$, $x^t(x-y)(x-z) + y^t(y-z)(y-x) + z^t(z-x)(z-y) \geq 0$. In this paper we study a generalization of this inequality to more terms, more general functions of the variables and algebraic structures such as vectors and Hermitian matrices.
\end{abstract}

\section{Introduction}
Issai Schur's classical inequality \cite{steele:maa2004} for the sum of products of the differences of real numbers\footnote{There are several inequalities attributed to Schur. For conciseness, in the sequel we refer to this inequality as simply Schur's inequality.} states that for $x,y,z,t\geq 0$
\begin{equation} \label{eqn:schur}
x^t(x-y)(x-z) + y^t(y-z)(y-x) + z^t(z-x)(z-y) \geq 0
\end{equation}

Because of the symmetry of Eq. (\ref{eqn:schur}), we can assume without loss of generality that $x\geq y\geq z\geq 0$. 
The purpose of this note is to consider further generalizations of Schur's inequality and extend it to more variables and other algebraic structures.

\section{Generalizations of Schur's inequality}
\begin{definition} 
A function $g:\R\rightarrow\R$
\begin{itemize}
\item is in class $\cG$ if $g$ is either even or odd, $g(0)\geq 0$, and $g$ is monotonically nondecreasing on $\R^+_0$.
\item is in class $\cG_2$ if $g$ is in class $\cG$ and for all $x\geq y\geq 0$ and $z\geq 0$,
\begin{eqnarray} 
g(x)g(y+z) &\geq& g(y)g(x+z)\label{eqn:condition1} \\
g(x)+g(y+z) &\leq& g(y)+g(x+z) \label{eqn:condition1b}
\end{eqnarray}
\end{itemize}\label{def:g}
\end{definition}
Note that Definition \ref{def:g} implies that if $g\in\cG$, then $g(x)\geq 0$ for $x\geq 0$.
A generalization of Schur's inequality is the following:
\begin{theorem} \label{thm:one}
Let $x\geq y \geq z$ and $a,c\geq 0$ such that $a+c \geq |b|$, and 
$g$ is a function in class $\cG$.
Then 
\begin{equation} \label{eqn:one}
ag(x-y)g(x-z) + bg(y-z)g(y-x) + cg(z-x)g(z-y) \geq 0
\end{equation} 
\end{theorem}
\begin{proof}
Let $r = x-y\geq 0$ and $s = y-z\geq 0$.
Then $ag(x-y)g(x-z) + bg(y-z)g(y-x) + cg(z-x)g(z-y) = ag(r)g(r+s)+bg(-r)g(s) + cg(-s-r)g(-s) = ag(r)g(r+s)+\alpha bg(r)g(s) + cg(s+r)g(s)$ where $\alpha = 1$ if $g$ is an even function and $\alpha = -1$ if $g$ is an odd function. 
The monotonicity of $g$ implies that this is larger than or equal to 
$ag(r)g(s) - |b|g(r)g(s) + cg(r)g(s) = (a+c-|b|)g(r)g(s)$ and the conclusion then follows.
\end{proof}

Schur's inequality considers products of the nontrivial differences between 3 variables $x$, $y$ and $z$. It is interesting to note that Schur's inequality can be considered as a consequence of the $2$ variable case, providing an alternative perspective of Schur's inequality. The 2 variable version of Theorem \ref{thm:one} is:

\begin{theorem}\label{thm:2vars}
Let $x\geq y$, $a\geq |b|$ and $g$ be a function in class $\cG$. Then
$ ag(x-y)+bg(y-x)\geq 0$.
\end{theorem}
Theorem \ref{thm:2vars} is trivially true since $g(x-y) = |g(y-x)|$. 
Let us now prove Theorem \ref{thm:one} using
Theorem \ref{thm:2vars}. 
If $c\geq |b|$, then $cg(z-x)g(z-y) = cg(y-z)g(x-z) \geq |bg(y-z)g(x-z)| \geq bg(y-z)g(y-x)$ and thus Eq. (\ref{eqn:one}) is satisfied.

Let $\tilde{a} = ag(x-z)\geq 0$, $\tilde{b}=bg(y-z)$ and $\tilde{c} = cg(y-z)\geq 0$. 
Suppose $|b|\geq c$. This implies that $|\tilde{b}|\geq \tilde{c}$.
$a+c\geq |b|$ implies that $\tilde{a}+\tilde{c}\geq |\tilde{b}|$.
Then
 $ag(x-y)g(x-z) + bg(y-z)g(y-x) + cg(z-x)g(z-y) = \tilde{a}g(x-y) + (\tilde{b}-\alpha\tilde{c})g(y-x) + c(g(z-x)g(z-y) + \alpha g(y-x)g(y-z))$, where $\alpha = \sign(b)$. 
 Since $c\geq 0$ and  $g(z-x)g(z-y)
\geq |g(y-x)g(y-z)|$, we see that the last term is nonnegative. Furthermore $|\tilde{b}-\alpha\tilde{c}| \leq |\tilde{b}|-\tilde{c} \leq \tilde{a}$ and the conclusion of Theorem \ref{thm:one} can be obtained by applying Theorem \ref{thm:2vars} to $\tilde{a}g(x-y) + (\tilde{b}-\alpha\tilde{c})g(y-x)$.

Even though this derivation of Theorem \ref{thm:one} from Theorem \ref{thm:2vars} is longer than the direct proof of Theorem \ref{thm:one} above, it is instructive as this observation of deducing the $(2n+1)$-variable case from the $2n$-variable case will be useful later on. 

\subsection{Extension of Schur's inequality to 4 variables}
 
\begin{theorem} \label{thm:gc}
Let $x_1\geq x_2\geq x_3\geq x_4$ be such that $x_1+x_4\geq x_2+x_3$. If $a_1 \geq \max(|a_2|,|a_4|)$,  $a_3 \geq |a_4|$ and $g$ is a function in $\cG$, then
\begin{align}
\sum_{i=1}^4a_i\prod_{i\neq j}g(x_i-x_j) \geq 0 \label{eqn:gc}
\end{align}
\end{theorem}

\begin{proof}
The left hand side of Eq. (\ref{eqn:gc}), denoted as $\beta$, can be written as
\begin{align*} \beta = & a_1g(x_1-x_4)g(x_1-x_2)g(x_1-x_3)+a_4g(x_4-x_1)g(x_4-x_3)g(x_4-x_2) \\
+ & a_2g(x_2-x_3)g(x_2-x_1)g(x_2-x_4) + a_3g(x_3-x_2)g(x_3-x_1)g(x_3-x_4)
\end{align*}
Using the fact that for $x\geq 0$, $|g(-x)| = g(x)$, we can bound this by:
\begin{align*} 
\beta \geq &  g(x_1-x_4)(a_1g(x_1-x_2)g(x_1-x_3)-|a_4|g(x_3-x_4)g(x_2-x_4)) \\
+ & g(x_2-x_3)(a_3g(x_1-x_3)g(x_3-x_4) -|a_2|g(x_1-x_2)g(x_2-x_4))\\
= & g(x_1-x_4)w_1+g(x_2-x_3)w_2
\end{align*}
where $w_1 = a_1g(x_1-x_2)g(x_1-x_3)-|a_4|g(x_3-x_4)g(x_2-x_4)$ and $w_2 = a_3g(x_1-x_3)g(x_3-x_4) -|a_2|g(x_1-x_2)g(x_2-x_4)$.
The hypothesis implies that $x_1-x_2\geq x_3-x_4$  Adding $x_2-x_3$ to both sides implies $x_1-x_3\geq x_2-x_4$. Thus $w_1\geq 0$.
If $w_2\geq 0$, then $\beta \geq 0$. Suppose $w_2 < 0$. Since $g(x_1-x_4)\geq g(x_2-x_3)$, this implies $g(x_1-x_4)w_2\leq g(x_2-x_3)w_2$, which means that
 $\beta \geq g(x_1-x_4)(w_1+w_2)$.  Next $w_1+w_2$ can be written as:
\begin{align*}
w_1+w_2 =&  g(x_1-x_2)(a_1g(x_1-x_3)-|a_2|g(x_2-x_4))  \\
+ &g(x_3-x_4)(a_3g(x_1-x_3)-|a_4|g(x_2-x_4))
\end{align*}
which is nonnegative since $a_1\geq |a_2|$, $a_3\geq |a_4|$ and $x_1-x_3\geq x_2-x_4$.
\end{proof}

Just as Theorem \ref{thm:one} can be deduced from Theorem \ref{thm:2vars},
a Corollary of Theorem \ref{thm:gc} is a generalization to 5 variables.

\begin{corollary} \label{cor:gc5}
Let $(x_i)_{i=1}^5$ be nonincreasing such that $x_1+x_4\geq x_2+x_3$. If $a_1 \geq \max(|a_2|,|a_4|-a_5)$, $a_3,a_5\geq 0$, $a_3+a_5 \geq |a_4|$, and $g$ is a function in $\cG$, then
\begin{align}
\sum_{i=1}^5a_i\prod_{i\neq j}g(x_i-x_j) \geq 0 \label{eqn:gc5}
\end{align}
\end{corollary}

\begin{proof}
If $a_5 \geq |a_4|$, then $a_5\prod_{j<5}g(x_5-x_j) = a_5\prod_{j<5}g(x_j-x_5)\geq 0$ can be written as the sum of $2$ nonnegative numbers $|a_4|\prod_{j<5}g(x_5-x_j) + (a_5-|a_4|)\prod_{j<5}g(x_5-x_j)$, so subtracting $ (a_5-|a_4|)\prod_{j<5}g(x_5-x_j)$ from the left hand side of Eq. (\ref{eqn:gc5}) can only decrease it. After subtraction, the coefficient $a_5$ will be equal to $|a_4|$ and will still satisfy the conditions in the hypothesis and $|a_4|\geq a_5$. Therefore without loss of generality we can assume that
$|a_4|\geq a_5\geq 0$.

Let use define the variables $\tilde{a}_i = a_ig(x_i-x_5)$ for $1\leq i\leq 4$ and $\tilde{a}_5 = a_5g(x_4-x_5)$. Then the hypothesis implies that
$\tilde{a}_1 \geq \max(|\tilde{a}_2|,|\tilde{a}_4|-\tilde{a}_5), \tilde{a}_3 \geq |\tilde{a}_4|-\tilde{a}_5\geq 0$. 
The left hand side of Eq. (\ref{eqn:gc5}) can be rewritten as
\begin{align}&\sum_{i=1}^3\tilde{a}_i\prod_{i\neq j,j<5}g(x_i-x_j)   + (\tilde{a}_4-\alpha\tilde{a}_5)\prod_{j=1}^3g(x_4-x_j)\nonumber\\&+a_5\left(\prod_{j\neq 5}g(x_5-x_j)+\alpha\prod_{j\neq 4}g(x_4-x_j)\right)\label{eqn:lhsgc5}
\end{align}
where $\alpha = \sign(a_4)$. Since  $a_5\geq 0$ and $\prod_{j\neq 5}g(x_5-x_j)\geq |\prod_{j\neq 4}g(x_4-x_j)|$, the last term in Eq. (\ref{eqn:lhsgc5}) is nonnegative.
By applying Theorem \ref{thm:gc} to $\sum_{i=1}^3\tilde{a}_i\prod_{i\neq j,j<5}g(x_i-x_j)   + (\tilde{a}_4-\alpha\tilde{a}_5)\prod_{j=1}^3g(x_4-x_j)$ the proof is complete.
\end{proof}

The procedure in the proof of Corollary \ref{cor:gc5} shows that Schur's inequalities of $2n$ variables can be used to derive Schur's inequality of $2n+1$ variables.

\subsection{Extension to 6 and 7 variables} \label{sec:67}
In this section we present Schur's inequalities for 6 and 7 variables. The following results give properties of functions in class $\cG_2$.

\begin{lemma}
If $f$ and $g$ are functions in class $\cG_2$, then so is the function $h$ defined by $h(x) = f(x)g(x)$.
If $g:\R\rightarrow\R$ is a function in class $\cG$ such that
$\frac{d^{2}\log(g)}{dx^2}\leq 0$ and $\frac{d^2g}{dx^2} \geq 0$ for $x\geq 0$, then $g$ is in class $\cG_2$.\label{lem:one}
\end{lemma}
\begin{proof}
$h$ satisfying Eq. (\ref{eqn:condition1}) is trivially true.
To show that $h$ satisfies Eq. (\ref{eqn:condition1b}) for $x\geq y \geq 0$ and $z\geq 0$,
\begin{align*}
h(x+z)-h(y+z) & = f(x+z)g(x+z)-f(y+z)g(y+z) \\ &= f(x+z)(g(x+z)-g(y+z)) + (f(x+z)-f(y+z))g(y+z) \\ 
& \geq f(x+z)(g(x)-g(y)) + (f(x)-f(y))g(y+z) \\ &\geq f(x)(g(x)-g(y)) + (f(x)-f(y))g(y) \\ &= f(x)g(x)-f(y)g(y) = h(x)-h(y)
\end{align*} 
Let $f(x) = \log(g(x))$. $\frac{d^2f}{dx^2} \leq 0$ implies that $f'(x)$ is monotonically nonincreasing and
$$f(x+z) - f(x) = \int_{x}^{x+z} f'(s)ds \leq \int_{y}^{y+z} f'(s)ds =  f(y+z)-f(y)$$
 and  $f(x)-f(y) \geq f(x+z)-f(y+z)$.
This is equivalent to $\frac{g(x)}{g(y)} \geq \frac{g(x+z)}{g(y+z)}$ and thus implying Eq. (\ref{eqn:condition1}). Similarly if $\frac{d^2g}{dx^2} \geq 0$, then $g$ satisfies Eq. (\ref{eqn:condition1b}) as $g'$ is monotonically nondecreasing.
\end{proof}

\begin{lemma}\label{lem:lemma3}
If $f$ is a function in class $\cG_2$ and $x\geq y\geq 0$, $z\geq w\geq 0$, $t\geq u\geq 0$, 
then $f(x+t)f(z+t) - f(x)f(z) \geq f(y+u)f(w+u) - f(y)f(w)$.
\end{lemma}
\begin{proof}
\begin{align*} 
&f(x+t)f(z+t)  - f(y+u)f(w+u) \\ &\geq f(x+u)f(z+u)  - f(y+u)f(w+u) \\
&= f(x+u)(f(z+u)-f(w+u))  + (f(x+u)-f(y+u))f(w+u) 
\\ & \geq f(x+u)(f(z)-f(w)) + (f(x)-f(y))f(w+u) \\ & \geq f(x)(f(z)-f(w)) + (f(x)-f(y))f(w)  \\ & =  f(x)f(z)-f(y)f(w)
\end{align*}
\end{proof}

Note that for $g(x)\neq 0$, $\frac{d^{2}\log(g)}{dx^2}\leq 0$ is equivalent to $gg'' \leq (g')^2$. Eq. (\ref{eqn:condition1}) implies that $g(x)^2 \geq g(x-z)g(x+z)$, i.e.  for $f= \log g$,
$f(x) \geq \frac{f(x-z)+f(x+z)}{2}$.

\begin{corollary} \label{cor:one}
The following functions and their pointwise products are in class $\cG_2$:
\begin{enumerate}
\item $g(x) = a$ for $a\geq 0$.
\item $g(x) = \sign(x) \stackrel{\makebox[0pt]{\mbox{\normalfont\tiny def}}}{=} \left\{ \begin{array}{ll} 1 & x > 0 \\ 0 & x = 0 \\ -1 & x < 0\end{array}\right.$.
\item $g(x) = |x|^s$ for $s\geq 1$.
\item $g(x) = e^{|x|}$.
\end{enumerate}
\end{corollary}
\begin{proof}
This can be shown using Lemma \ref{lem:one} and taking the second derivative of $g$ and of $\log g$.
For the case where $g(x) = |x|^s$, $\frac{d^{2}\log(g)}{dx^2}$ is not defined for $x=0$. However, $g(0)= 0$ and 
Eq. (\ref{eqn:condition1}) is satisfied when $y= 0$ and we can apply Lemma \ref{lem:one} to the cases where $y\neq 0$.
\end{proof}

Note that Corollary \ref{cor:one} implies that the power functions  $f(x) = x^k$ for integer $k \geq 0$ are in class $\cG_2$. This is due to the fact that $x^k = |x|^k$ for $k$ even and $x^k = \sign(x)|x|^k$ for $k$ odd.

\begin{theorem} \label{thm:gc6}
Let $(x_i)$ for $1\leq i\leq 6$ be nonincreasing and $x_1+x_6\geq x_2+x_5 \geq x_3+x_4$. If $a_1 \geq |a_2|\geq a_5\geq |a_6|$,
 $a_3\geq |a_4|$,  and $g$ is a function in $\cG_2$, then
\begin{align}
\sum_{i=1}^6 a_i\prod_{j\neq i}g(x_i-x_j) \geq 0 \label{eqn:gc6}
\end{align}
\end{theorem}

\begin{proof}
Define  $g_{i,j} = g(x_i-x_j)$ and $\gamma_i = \prod_{j\neq i}g_{i,j}$. Then the left hand side of Eq. (\ref{eqn:gc6}) can be written as $\sum_{i=1}^6 a_i\gamma_i$.
Note that $\sum_i a_i\gamma_i \geq a_1\gamma_1 - |a_2||\gamma_2| + \cdots - |a_6||\gamma_6|$. 
It is easy to see that $\gamma_1\geq |\gamma_2|$.

Consider the 2 terms $\gamma_3$ and $\gamma_4$.
Since $g\in \cG_2$ and $x_1-x_3\geq x_4-x_6$, by setting
$x = x_1-x_3$, $y = x_4-x_6$ and $z = x_3-x_4$ in Eq. (\ref{eqn:condition1}), we see that $g_{1,3}g_{3,6} \geq g_{1,4}g_{4,6}$. 

Similarly by setting
$x = x_2-x_3$, $y = x_4-x_5$ and $z = x_3-x_4$ in Eq. (\ref{eqn:condition1}),  we see that $x_2-x_3\geq x_4-x_5$ implies $g_{2,3}g_{3,5}\geq g_{2,4}g_{4,5}$. This means that $a_3\gamma_3-|a_4||\gamma_4| \geq a_3(\gamma_3-|\gamma_4|) \geq 0$. 

Since $x_1-x_2 \geq x_5-x_6$ by setting $x = x_1-x_2$, $y = x_5-x_6$ and $z = x_2-x_5$ we see that $g_{1,2}g_{2,6}\geq g_{1,5}g_{5,6}$. $x_2-x_4\geq x_3-x_5$ implies that $g_{2,4}\geq g_{3,5}$ and  $x_2-x_3\geq x_4-x_5$ implies that $g_{2,3}\geq g_{4,5}$, i.e. $g_{2,3}g_{2,4} \geq g_{3,5}g_{4,5}$. This shows that $|\gamma_2|\geq \gamma_5$.

Next we show that $a_1\gamma_1+a_5\gamma_5 \geq |a_2||\gamma_2|+|a_6||\gamma_6|$. Let us define $\eta =g_{1,5}g_{1,6}g_{2,6}\geq 0$.
 
Since $x_2-x_6 \leq x_1-x_5$, this implies that $g_{2,6}\leq g_{1,5}$ and
$g_{2,6}a_5\gamma_5 -g_{1,5}|a_2||\gamma_2| \leq g_{2,6}(a_5\gamma_5-|a_2||\gamma_2|) \leq 0$. 
Since $g_{1,6}\geq g_{2,5}$,
it is straightforward to show that
\begin{align}
0 &\geq g_{2,6}(a_5\gamma_5-|a_2||\gamma_2|) \geq g_{2,6}a_5\gamma_5 -g_{1,5}|a_2||\gamma_2| \nonumber \\
& = g_{1,5}g_{2,5}g_{2,6}\left(a_5g_{3,5}g_{4,5}g_{5,6}-|a_2|g_{1,2}g_{2,3}g_{2,4}\right) \geq \eta\left(a_5g_{3,5}g_{4,5}g_{5,6}-|a_2|g_{1,2}g_{2,3}g_{2,4}\right) 
\label{eqn:gc6-eq1}
\end{align}
Similarly
\begin{align}
&g_{2,6}(a_1\gamma_1-|a_6||\gamma_6|) \geq  g_{2,6}a_1\gamma_1-g_{1,5}|a_6||\gamma_6| \nonumber \\  
& = g_{1,5}g_{1,6}g_{2,6} \left(a_1g_{1,2}g_{1,3}g_{1,4} -|a_6|g_{3,6}g_{4,6}g_{5,6}\right) =
\eta\left(a_1g_{1,2}g_{1,3}g_{1,4} -|a_6|g_{3,6}g_{4,6}g_{5,6} \right)
\label{eqn:gc6-eq2}
\end{align}
We add Eq. (\ref{eqn:gc6-eq1}) to Eq. (\ref{eqn:gc6-eq2}) to obtain
\begin{align*} &g_{2,6}\left(a_1\gamma_1+a_5\gamma_5-|a_2||\gamma_2|-|a_6||\gamma_6|\right)
\\ & \geq \eta(a_5g_{3,5}g_{4,5}g_{5,6}-|a_2|g_{1,2}g_{2,3}g_{2,4}+a_1g_{1,2}g_{1,3}g_{1,4}-|a_6|g_{3,6}g_{4,6}g_{5,6})
\end{align*}
The right hand side of the equation above can be written as $\eta\delta$ where
\begin{equation}\label{eqn:term1}
\delta  = g_{5,6}\left(a_5g_{3,5}g_{4,5}-|a_6|g_{3,6}g_{4,6}\right) + g_{1,2}\left(a_1g_{1,3}g_{1,4}-|a_2|g_{2,3}g_{2,4}\right)
\end{equation}
  Since $a_3\gamma_3\geq |a_4|\gamma_4$, the conclusion follows if $\delta \geq 0$. Let $\mu = a_5g_{3,5}g_{4,5}-|a_6|g_{3,6}g_{4,6}$. Note that $g_{1,2}\left(a_1g_{1,3}g_{1,4}-|a_2|g_{2,3}g_{2,4}\right) \geq 0$. If $\mu \geq 0$, then $\delta \geq 0$.
For the case $\mu \leq 0$, $a_5\geq |a_6|$ implies that $\mu \geq a_5(g_{3,5}g_{4,5}-g_{3,6}g_{4,6})$. Since $g_{1,2}\geq g_{5,6}$ and $a_1\geq a_5$,
\begin{equation*}
0 \geq g_{5,6}\mu \geq g_{1,2}\mu \geq a_5g_{1,2}\left(g_{3,5}g_{4,5}-g_{3,6}g_{4,6}\right) \geq a_1g_{1,2}\left(g_{3,5}g_{4,5}-g_{3,6}g_{4,6}\right)
\end{equation*}
Similarly $a_1g_{1,4}g_{1,3}-|a_2|g_{2,3}g_{2,4}
\geq a_1(g_{1,4}g_{1,3}-g_{2,3}g_{2,4})$. This means that
\begin{align*}
\delta & \geq a_1g_{1,2}(g_{1,3}g_{1,4}+g_{3,5}g_{4,5}-g_{2,3}g_{2,4}-g_{3,6}g_{4,6})
\end{align*}
Lemma \ref{lem:lemma3} implies that $g_{1,3}g_{1,4}+g_{3,5}g_{4,5}-g_{2,3}g_{2,4}-g_{3,6}g_{4,6}\geq 0$ by choosing $x=x_2-x_3$, $z=x_2-x_4$, $w=x_3-x_5$, $y=x_4-x_5$, $t = x_1-x_2$ and $u = x_5-x_6$ and this implies that $\delta \geq 0$.
\end{proof}

Similar to Corollary \ref{cor:gc5}, we have:
\begin{corollary} \label{cor:gc7}
Let $(x_i)$ for $1\leq i\leq 7$ be nonincreasing and $x_1+x_6\geq x_2+x_5 \geq x_3+x_4$. If $a_1 \geq |a_2|\geq a_5\geq |a_6|-a_7$,
 $a_3\geq |a_4|$, $a_5, a_7\geq 0$  and $g$ is a function in $\cG_2$, then
\begin{align}
\sum_{i=1}^7 a_i\prod_{j\neq i}g(x_i-x_j) \geq 0 \label{eqn:gc7}
\end{align} 
\end{corollary}
\begin{proof}
As in the proof of Theorem \ref{thm:gc6}, we define $g_{i,j} = g(x_i-x_j)$.
If $a_7 \geq |a_6|$, then $a_7\prod_{j<7}g_{7,j}$ can be written as the sum of $2$ nonnegative numbers $|a_6|\prod_{j<7}g_{7,j} + (a_7-|a_6|)\prod_{j<7}g_{7,j}$, so subtracting $ (a_7-|a_6|)\prod_{j<7}g_{7,j}$ from the left hand side of Eq. (\ref{eqn:gc7}) can only decrease it. Therefore, similar to Corollary  \ref{cor:gc5} we can assume that
$|a_6|\geq a_7\geq 0$ without loss of generality.

Let us define the variables $\tilde{a}_i = a_ig_{i,7}$ for $1\leq i\leq 6$ and $\tilde{a}_7 = a_7g_{6,7}$. Then the hypothesis implies that
$\tilde{a}_1 \geq |\tilde{a}_2| \geq \tilde{a}_5\geq |\tilde{a}_6|-\tilde{a}_7$, $\tilde{a}_3\geq |\tilde{a}_4|$,$\tilde{a}_5\geq 0$. 
The left hand side of Eq. (\ref{eqn:gc7}) can be rewritten as
\begin{equation}\sum_{i=1}^5\tilde{a}_i\prod_{i\neq j,j<7}g_{i,j}   + (\tilde{a}_6-\alpha\tilde{a}_7)\prod_{j=1}^5g_{6,j}+a_7\left(\prod_{j\neq 7}g_{7,j}+\alpha\prod_{j\neq 6}g_{6,j}\right)\label{eqn:lhsgc7}
\end{equation}
where $\alpha = \sign(a_6)$. Since  $a_7\geq 0$ and $\prod_{j\neq 7}g_{7,j}\geq |\prod_{j\neq 6}g_{6,j}|$, the last term in Eq. (\ref{eqn:lhsgc7}) is nonnegative.
By applying Theorem \ref{thm:gc6} to $\sum_{i=1}^5\tilde{a}_i\prod_{i\neq j,j<7}g_{i,j}   + (\tilde{a}_6-\alpha\tilde{a}_7)\prod_{j=1}^5g_{6,j}$ the proof is complete.
\end{proof}

\subsection{$Q$-class functions} \label{sec:q-class}
In Eq. (\ref{eqn:one}), $x^t$, $y^t$ and $z^t$ in Schur's inequality are replaced with real numbers $a$, $b$ and $c$ respectively that do not depend on $x$, $y$ and $z$.
Over the years, Schur's inequality in Eq. (\ref{eqn:schur}) has been generalized in various ways \cite{mitrinovic1990,Mitrinovic1993,grinberg:2007,Radulescu2009,BullenL2015} and in many cases $x^t$, $y^t$ and $z^t$ are replaced with general functions of $x$, $y$ and $z$ respectively. For instance, in Ref. \cite{Wright1956} we find the generalization:
\begin{theorem}
Let $x\geq y \geq z$, and $f$ a positive function that is either convex, monotonic or $\max_x f(x) < 4\min_x f(x)$. Then for nonnegative integer $k$,
\begin{equation}  \label{eqn:wright}
f(x)(x-y)^k(x-z)^k + f(y)(y-z)^k(y-x)^k + f(z)(z-x)^k(z-y)^k \geq 0
\end{equation}
\label{thm:wright}
\end{theorem}

Suppose $a\geq b\geq c$. If $f$ is monotonic, then it is clear that either $f(a)\geq f(b)$ or $f(c)\geq f(b)$. If $f$ is convex then $f(b)\leq \alpha f(a) + (1-\alpha)f(c) \leq \max(f(a),f(c))$ for $b=\alpha a+(1-\alpha)c$, so again either $f(a)\geq f(b)$ or $f(c)\geq f(b)$. If $f$ is a nonnegative function, then $f(a)+f(c) \geq f(b) = |f(b)|$. 
Thus Theorem  \ref{thm:one} implies the following result which is a generalization of Theorem \ref{thm:wright}.
\begin{corollary}\label{cor:cor-one}
Let $x\geq y\geq z$, $g\in \cG$ and $f$ a nonnegative monotonic or convex function. Then
\begin{equation}  \label{eqn:q}
f(x)g(x-y)g(x-z) + f(y)g(y-z)g(y-x) + f(z)g(z-x)g(z-y) \geq 0
\end{equation}
\end{corollary}

The following definiton of $Q$-class was first introduced in Ref. \cite{godunova-levin:1985}. 
\begin{definition} \label{def:q-class}
A function $f$ is in $Q$-class if $f(\lambda x+(1-\lambda)y) \leq \frac{f(x)}{\lambda} + \frac{f(y)}{1-\lambda}$ for $\lambda \in (0,1)$.
\end{definition}
It was shown by Godunova and Levin that the set of $Q$-class functions contains all nonnegative convex or monotone functions and that 
\begin{equation}  \label{eqn:wright}
f(x)(x-y)(x-z) + f(y)(y-z)(y-x) + f(z)(z-x)(z-y) \geq 0
\end{equation}
for all $x$, $y$ and $z$ if and only if $f$ is in $Q$-class, i.e. Eq. (\ref{eqn:q}) is satisfied for $g(x) = x$.
However if $f$ is in $Q$-class, but $f$ is not convex nor monotone, then it does not necessarily imply that Eq. (\ref{eqn:q}) holds for all function $g\in\cG$. For instance if $f$ is a nonconstant $Q$-class function defined on $[0,1]$ such that $f(\frac{1}{2}) = 2(f(0)+f(1))$, then it is easy to verify that  Eq. (\ref{eqn:q}) does not hold when
$g(x) = \sign(x)\sqrt{|x|}$ by choosing $x=1$, $y=\frac{1}{2}$ and $z=0$ in Eq. (\ref{eqn:q}). 
An example of such a function is $f(x) = 4$ if $x=\frac{1}{2}$ and $f(x) = 1$ otherwise. Clearly $f$ is not constant with $f(\frac{1}{2}) = 2(f(0)+f(1))$.
To show that $f$ is in $Q$-class, note that $\frac{1}{\lambda}+\frac{1}{1-\lambda} \geq 4$ for $\lambda\in (0,1)$.
Since $f(x)\geq 1$ for all $x\in [0,1]$, this means that $\frac{f(x)}{\lambda} + \frac{f(y)}{1-\lambda} \geq 4 \geq f(\lambda x+ (1-\lambda)y)$.

\begin{definition}
Let $f, g$ be functions from $\mathbb{R}$ to $\mathbb{R}$. A function $f$  is in $Q_g$-class if 
\begin{equation} \label{eqn:extended_q}
f(\lambda x+ (1-\lambda)z) \leq \frac{f(x)g(x-z)}{g(\lambda(x-z))} +  \frac{f(z)g(x-z)}{g((1-\lambda)(x-z))} 
\end{equation}
for all $x$, $z$ and $\lambda \in (0,1)$ such that $g(\lambda(x-z))\neq 0$ and $g((1-\lambda)(x-z))\neq 0$.
\end{definition}

The following theorem extends the $Q$-class condition on $f$ so that Eq. (\ref{eqn:q}) is satisfied for arbitrary functions $g$ in $\cG$.

\begin{theorem} \label{thm:extended_q}
Let $g$ be an odd function in class $\cG$. Then Eq. (\ref{eqn:q}) is satisfied 
for all $x\geq y\geq z$ if and only if $f$ is in $Q_g$-class.
\end{theorem}
\begin{proof}
The proof is simply rewriting Eq. (\ref{eqn:q}) as Eq. (\ref{eqn:extended_q}) via
$y = \lambda x + (1-\lambda) z$.
\end{proof}

Note that Godunova and Levin's result is the special case when $g(x) = x$.

\begin{corollary} \label{cor:extended q}
Let $k\geq 0$ be a real number and $g(x) = \sign(x)|x|^k$. Then Eq. (\ref{eqn:q}) is satisfied 
for all $x\geq y\geq z$ if and only if $f$ satisfies
\begin{equation} \label{eqn:new_q}
f(\lambda x+ (1-\lambda)z) \leq \frac{f(x)}{\lambda^k} +  \frac{f(z)}{(1-\lambda)^k} 
\end{equation}
for all $\lambda \in (0,1)$ and $x\neq z$.
\end{corollary}

Note that for $k\geq 1$, Eq. (\ref{eqn:new_q}) is satisfied if $f$ is in $Q$-class, but not necessarily vice versa. Similarly for $0\leq k\leq 1$, if Eq. (\ref{eqn:new_q}) is satisfied then $f$ is in $Q$-class, but not necessarily vice versa.
In Ref. \cite{vornicu} a variation of Schur's inequality is presented:
\begin{theorem}
Let $x\geq y \geq z$ and either $a\geq b\geq c$ or $a\leq b \leq c$, $k> 0$ an integer and $f$ is either convex or monotonic such that $f(x) \geq 0$. Then
\begin{equation}  \label{eqn:vornicu}
f(a)(x-y)^k(x-z)^k + f(b)(y-z)^k(y-x)^k + f(c)(z-x)^k(z-y)^k \geq 0
\end{equation}
\label{thm:vornicu}
\end{theorem}

Note that in Theorem \ref{thm:vornicu} the condition $a\leq b \leq c$  is redundant as it is equivalent to the case $a\geq b\geq c$. This is because $f$ is monotonic (resp. convex) if and only if $x \mapsto  f(-x)$ is monotonic (resp. convex).
The argument above regarding Eq. (\ref{eqn:q}) and the fact that $x \mapsto  x^k$ is in class $\cG$ shows that Theorem \ref{thm:one} generalizes Theorem \ref{thm:vornicu}.
In addition, the convex part of Theorem \ref{thm:vornicu} can be generalized as:
\begin{theorem}
Let $x\geq y \geq z$ and $a\geq c$, $k> 0$ an integer and $f$ is a convex function such that $f(x) \geq 0$. Then
\begin{equation}  \label{eqn:vornicugeneral}
\alpha f(a)(x-y)^k(x-z)^k + f(b)(y-z)^k(y-x)^k + (1-\alpha)f(c)(z-x)^k(z-y)^k \geq 0
\end{equation}
for $\alpha\in [0,1]$ and $b=\alpha a+(1-\alpha)c$.
\label{thm:vornicugeneral}
\end{theorem}

In Ref. \cite{finta:2015}, Schur's inequality is extended to 4 variables.
\begin{theorem}
Let $x_1\geq x_2\geq x_3\geq x_4\geq 0$ and $t > 0$ such that $x_1+x_4\geq x_2+x_3$, then
\begin{align}
\sum_{i=1}^4x_i^t\prod_{i\neq j}(x_i-x_j) \geq 0 \label{eqn:finta}
\end{align}
\label{thm:finta}
\end{theorem}
Ref. \cite{finta:2015} also gave an example where Eq. (\ref{eqn:finta}) does not hold if  $x_1+x_4 < x_2+x_3$. 
Theorem \ref{thm:finta} is  a special case of Theorem \ref{thm:gc}  where $x_4\geq 0$, $g(x) = x$ and $a_i = x_i^t$ for $t > 0$.
Similarly, Theorem \ref{thm:gc6} generalizes the result in Ref. \cite{finta:2018} which considered the special case of $x_6\geq 0$, $g(x) = x$ and $a_i=x_i^t$ for $t > 0$.
\section{Extension of Schur's inequality to other algebraic structures} \label{sec:structures}
So far the examples above are about functions of real numbers. In this section we look at other partially ordered sets for which equations such as Eq. (\ref{eqn:one}), Eq. (\ref{eqn:gc}), Eq. (\ref{eqn:gc5}), Eq. (\ref{eqn:gc6}) and Eq. (\ref{eqn:gc7}) can be deduced.

\begin{definition}[\cite{mathdict:1993}]
A partially ordered group $(G,+,\preceq)$ is defined as a group $G$ with group operation $+$ and a partial order $\preceq$ on $G$ such that 
 $z+x\preceq z+y\Leftrightarrow x+z \preceq y+z \Leftrightarrow x\preceq y$ for all $x,y,z\in G$. 
\end{definition}

For a partially order group $(G,+,\preceq)$, the set $\{x\in G:x\succeq 0\}=G^+$ is called the positive cone.

\begin{definition}[\cite{mathdict:1993}]
$(G,+,\preceq)$ is a partially ordered real vector space if $(G,+)$ is a  partially ordered group, $G$ is a real vector space 
and $x\preceq y$ implies $\lambda x \preceq \lambda y$ for all real $\lambda\geq 0$.
\end{definition}

\begin{definition} \label{def:cg}
Let $\cal C$ be defined as the set of tuples $\ijktuple$ satisfying the following conditions:
\begin{enumerate}
\item $(I,+_I,\preceq_I)$,  $(J,+_J,\preceq_J)$ and  $(K,+_K,\preceq_K)$ are partially ordered Abelian groups.
\item $\ast : I\times J\rightarrow K$ is a {\em distributive} operation, i.e. it satisfies $(x+_Iy)\ast z = x\ast z +_K y\ast z$ and $x\ast (y+_J z) = x\ast y +_K x\ast z$.
\item $\ast$ is nonnegativity-preserving: if $x\succeq_I 0$ and $y\succeq_J 0$, then $x\ast y\succeq_K 0$.
\end{enumerate}
\end{definition}
Note that the distributive property implies  $0_I\ast x = y\ast 0_J = 0_K$ since $x\ast y = (0_I+_Ix)\ast y = 0_I\ast y +_K x\ast y$ implies $0_I\ast y = 0_K$.
It is perhaps unusual to denote functions from $I\times J$ to $K$ with an infix notation $\ast$, but it is a useful notation to denote the distributive properties above and one must keep in mind that $(x\ast y)\ast z$ is in general undefined. However, when $x \ast x \ast \cdots \ast x$ is well-defined, we will denote it as $x^k$. 

Examples of elements in $\cal C$ are listed in Table \ref{tbl:calc}. In these examples, $I=J$ and $K$ is a partially ordered real vector space. These structures have been found useful in generalizing rearrangement inequalities \cite{wu:2020}. It is easy to verify that these tuples satisfy Definition \ref{def:cg}. For instance, for the case where $A\ast B$ is the Frobenius inner product $Tr(A^HB)$, it is well known that the trace of the product of Hermitian positive (semi)-definite matrices is positive (nonnegative) \cite{Coope1994}. Similarly the Schur product theorem states that the Hadamard product of two positive definite matrices is also positive definite. The Kronecker product of two positive definite matrices $A$ and $B$ is positive definite since the eigenvalues of the Kronecker product $A\otimes B$ are formed from the products of the eigenvalues of $A$ and $B$ (see also Ref. \cite{Horn1991}).

\begin{table}[htbp]
\begin{center}
\begin{tabular}{|c|c|c|c|c|c|}
\hline
$I=J$& $\preceq_I$&  $K$& $\preceq_K$& $\ast$&\shortstack{symmetric \\$\ast$}\\
\hline\hline
$\mathbb{R}$ & $\leq$ &  $\mathbb{R}$ & $\leq$  & multiplication & yes \\
\hline
$\mathbb{R}^n$ & \shortstack{induced by\\positive\\cone} &  $\mathbb{R}$ & $\leq$ & dot product & yes \\
\hline
$\mathbb{R}^n$ &  \shortstack{induced by\\positive\\cone}  & $\mathbb{R}$ & $\leq$ & \shortstack{$x\ast y = x^TAy$ \\with $A \geq 0$} & \shortstack{no\\yes if \\$A=A^T$} \\
\hline
$f:[0,1]\rightarrow\mathbb{R}$ & \shortstack{induced by\\positive\\cone} &  $\mathbb{R}$ & $\leq$ & \shortstack{$f\ast g =$\\$\int_{0}^1 f(x)g(x) dx$} & yes \\
\hline
$\mathbb{R}^{n\times n}$ &  \shortstack{induced by\\positive\\cone}  &  $\mathbb{R}^{n\times n}$ &  \shortstack{induced by\\positive\\cone} & \shortstack{Matrix \\multiplication} & no \\
\hline
\shortstack{Hermitian\\ matrices} &  \shortstack{ Loewner\\ order}  &  $\mathbb{R}$ & $\leq$ & \shortstack{Frobenius\\ inner\\ product} & yes \\
\hline
\shortstack{Commuting \\Hermitian\\ matrices}  & \shortstack{ Loewner\\ order}    & \shortstack{Hermitian\\ matrices}  &  \shortstack{ Loewner\\ order}  & \shortstack{Matrix \\multiplication} & yes \\
\hline
\shortstack{Hermitian\\ matrices} &  \shortstack{ Loewner\\ order}    &  \shortstack{Hermitian\\ matrices}& \shortstack{ Loewner\\ order} & \shortstack{Hadamard\\ product $\circ$} & yes \\
\hline
\shortstack{Hermitian\\ matrices} &  \shortstack{ Loewner\\ order}  &  \shortstack{Hermitian\\ matrices}& \shortstack{ Loewner\\ order} & \shortstack{Kronecker\\ product $\otimes$} & no \\
\hline
\shortstack{Hermitian\\ matrices} &  \shortstack{ Loewner\\ order}  &  \shortstack{Hermitian\\ matrices}& \shortstack{ Loewner\\ order} & \shortstack{reverse Kronecker \\ product\tablefootnote{The reverse Kronecker product is defined as $A\otimes_r B$ = $B\otimes A$.} $\otimes_r$}& no \\
\hline

\end{tabular}
\end{center}
\caption{Examples of members in $\cal C$. In these examples $I=J$.}\label{tbl:calc}
\end{table}

Definition \ref{def:g} is also applicable to structures in ${\cal C}$.
\begin{definition} 
Let $(I,+_I,\preceq_I), (J,+_J,\preceq_J)$ be partially ordered groups.
A function $g:I\rightarrow J$ is in class $\cG$ if $g(0)\succeq_J 0_J$, $x\succeq_I y\succeq_I 0_I \rightarrow g(x)\succeq_J g(y)$ and either $\forall x\in I, g(x) =g(-x)$ or $\forall x\in I, g(x)=-g(-x)$. 

Let $\ijktuple\in {\cal C}$ with $(I,+_I,\preceq_I) = (J,+_J,\preceq_J)$.
A function $g:I\rightarrow I$ is in class $\cG_2$ if $g$ is in class $\cG$, and for all $x\succeq_I y\succeq_I 0_I$ and $z\succeq_I 0$,
\begin{eqnarray} 
g(x)\ast g(y+z) &\succeq_K& g(y)\ast g(x+z)\label{eqn:condition1C} \\
g(x)+_I g(y+z) &\preceq_I& g(y)+_I g(x+z) \label{eqn:condition1bC}
\end{eqnarray}
\label{def:gC}
\end{definition}

\begin{lemma} \label{lem:gC}
Let $\ijktuple$ be a tuple in $\cal C$,  $A\succeq_I B \succeq_I 0$ and $C \succeq_J D \succeq_J 0$.
\begin{itemize}
\item $A\ast C\succeq_K B\ast D\succeq_K 0$.
\item  If in addition, $I = J = K$, $\preceq_I = \preceq_J = \preceq_K$, $+_I=+_J=+_K=+$, then $A^n\succeq_I B^n\succeq_I 0$
for all integers $n \geq 0$.
\item If in addition $I$ is a partially ordered real vector space, then $p(A)\succeq_I p(B)\succeq_I 0$ for polynomials $p$ with nonnegative real coefficients. 
\item
If in addition $\ast$ is symmetric and associative, then  $A^n\ast (B+C)^n \succeq_I B^n \ast (A+C)^n \succeq_I 0$ and $(A+C)^n+B^n\succeq_I (B+D)^n +A^n$  for all integers $n\geq 0$.
\end{itemize}
\end{lemma}

\begin{proof}
 If $A\succeq_I B \succeq_I 0$ and
$C \succeq_J D \succeq_J 0$, then $(A-B)\ast C \succeq_K 0$ implies that $A\ast C \succeq_K B\ast C$. This implies that $A\ast C \succeq_K B\ast C \succeq_K B\ast D\succeq 0$. 
If $I = J = K$, $\preceq_I = \preceq_J = \preceq_K$, then this implies that $A^n\succeq_I B^n\succeq_I 0$. Similarly if $I$ is a real vector space and $\alpha\geq 0$, then $\alpha A^n \succeq_I \alpha B^n$ and thus $p(A)\succeq_I p(B)$.
Finally, if $\ast$ is symmetric and associative, the terms of $A^n\ast(B+C)^n$ and the corresponding terms of $B^n\ast(A+C)^n$ satisfy the relationship $\binom{n}{k}A^n\ast B^k\ast C^{n-k} \succeq_I \binom{n}{k}A^k\ast A^{n-k}\ast B^k\ast C^{n-k} \succeq_I \binom{n}{k}A^k\ast B^{n-k}\ast B^k\ast C^{n-k}
\succeq_I  \binom{n}{k}B^n\ast A^k\ast C^{n-k}$ when $0\leq k\leq n$. Similarly,  $(A+C)^n+B^n = \sum_{k=0}^{n-1} \binom{n}{k}A^k \ast C^{n-k} + A^n +  B^n$ and $(B+D)^n+A^n = \sum_{k=0}^{n-1} \binom{n}{k}B^k \ast D^{n-k} + A^n +  B^n$ and for each $0\leq k \leq n-1$, the corresponding terms satisfy  $\binom{n}{k}A^k \ast C^{n-k} \succeq_I \binom{n}{k}B^k \ast D^{n-k}$.
\end{proof}

The conditions $A^n\ast (B+C)^n \succeq_I B^n \ast (A+C)^n \succeq_I 0$ and $(A+C)^n+B^n\succeq_I (B+D)^n +A^n$ in  Lemma \ref{lem:gC} show that the power functions  $x\rightarrow  x^n$ for integers $n\geq 0$ satisfy the conditions in Eqs. (\ref{eqn:condition1C})-(\ref{eqn:condition1bC}) in Definition  \ref{def:gC} that the nonlinear function $g$ in Theorem \ref{thm:gc6} and Corollary \ref{cor:gc7} needs to satisfy. 

The special case $I=J=K$ is important and useful for products of more than 2 terms.
\begin{definition} \label{def:cg2}
Let ${\cal C}_2$ be defined as the set of tuples $\ituple$ such that $(I,+,\preceq_I,I,+,\preceq_I,I,+,\preceq_I,\ast)$ is in $\cal C$.
\end{definition}

If $(I,+,\succeq_I))$ is a partially ordered group with an associative, distributive and nonnegativity preserving operation $\ast : I\times I \rightarrow I$ whose identity is in $I$, then $\ituple$ is a partially ordered ring. If in addition $\ast$ is commutative, then $\ituple$ is a partially ordered commutative ring.
Lemma \ref{lem:gC} along with 
the proof of Lemma \ref{lem:lemma3} can be used to show the following result.
\begin{lemma}\label{lem:lemma3G}
If $(I,+,\succeq_I,\ast)$ is a partially ordered commutative ring and $x\succeq_I y\succeq_I 0$, $z\succeq_I w\succeq_I 0$, $t\succeq_I u\succeq_I 0$, 
then $(x+t)^n\ast (z+t)^n - x^n\ast z^n \succeq_I (y+u)^n\ast (w+u)^n - y^n\ast w^n$ for all integers $n\geq 0$.
\end{lemma}

\begin{definition}
For a partial ordered set $(I,\succeq_I)$, $u$ is an upper bound of $S\subset I$ if $u\succeq_I s$ for all $s\in S$. A greatest element of $S\subset I$ is an element $g\in S$ such that $g\succeq_I s$ for all $s\in S$.
\end{definition}

All rows in Table \ref{tbl:calc} except for rows 2,3,4,6 are partially ordered rings.
For the last 2 rows in Table \ref{tbl:calc}, since the Kronecker product $A\otimes B$ are of higher order than the order of $A$ and $B$, in order to define it as an element of ${\cal C}_2$, the set $I$ needs to be defined as the set of Hermitian matrices of {\em all} orders\footnote{Another interpretation of the Kronecker product of matrices is a case where $I\neq J\neq K$. For instance, for fixed integers $n, m > 0$, $I$, $J$ and $K$ can be chosen as the set of $n\times n$, $m\times m$ and $nm\times nm$ Hermitian matrices respectively.}.
We have the following generalization of Theorem \ref{thm:one}.
\begin{theorem} \label{thm:general_one}
Let $x\succeq_I y \succeq_I z$ and $n\geq 0$ an integer and $p\in \cal G$.
\begin{itemize}
\item Let $\ijktuple\in \cal C$, $a,b,c\in \R$ such that $a,c\geq 0$, $a+c\geq |b|$ and $K$ is a partially ordered real vector space,
If $I=J$, $\preceq_I = \preceq_J$, 
then 
\begin{equation} \label{eqn:oneC1}
a (x-_Iy) \ast (x-_Jz) + b (y-_Ix)\ast (y-_Jz) + c (z-_Ix)\ast (z-_Jy) \succeq_K 0
\end{equation} 
\item If $\ituple \in {\cal C}_2$, $a,c\succeq_I 0$,  and $a+c$ is an upper bound of $\{b,-b\}$, then 
\begin{equation} \label{eqn:oneC2}
a\ast (x-y)^n \ast (x-z)^n + b \ast (y-x)^n\ast (y-z)^n + c \ast (z-x)^n\ast (z-y)^n \succeq_I 0
\end{equation}
\item If $\ituple \in {\cal C}_2$, $a,c\succeq_I 0$,  $a+c$ is an upper bound of $\{b,-b\}$, and $I$ is a partially ordered real vector space, then
\begin{equation} \label{eqn:oneC3}
a\ast p(x-y) \ast p(x-z) + b \ast p(y-x)\ast p(y-z) + c \ast p(z-x)\ast p(z-y) \succeq_I 0
\end{equation}
\end{itemize}
\end{theorem}

\begin{proof}
We will only give a sketch of the proof as it is similar to the proof of Theorem \ref{thm:one}.
The left hand side of  Eq. (\ref{eqn:oneC1}) can be written as
$$a (x-_Iy) \ast (x-_Jz) - b (x-_Iy)\ast (y-_Jz) + c (x-_Iz)\ast (y-_Jz)$$
By Lemma \ref{lem:gC}, $(x-_Jz)\succeq_K (y-_Jz)$ implies that $a (x-_Iy) \ast (x-_Jz)\succeq_K a(x-_Iy)\ast (y-_Jz)$.
 Similarly, $c (x-_Iz)\ast (y-_Jz) \succeq _K c(x-_Iy)\ast (y-_Jz)$
 and thus $a (x-_Iy) \ast (x-_Jz) + c (x-_Iz)\ast (y-_Jz) \succeq_K (a+c)(x-_Iy)\ast (y-_Jz) \succeq_K |b|(x-_Iy)\ast (y-_Jz)$
implying
Eq. (\ref{eqn:oneC1}). 
Similarly $(x-_Jz)\succeq_K (y-_Jz)$ implies $(x-_Jz)^n\succeq_K (y-_Jz)^n$ which in turn implies that $(x-_Iy)^n \ast (x-_Jz)^n \succeq_K (x-_Iy)^n\ast (y-_Jz)^n$. This in addition with the fact that $b\ast (y-x)^n\ast(y-z)^n \preceq_I  (a+c)\ast (x-y)^n\ast(y-z)^n$  can be used to show 
Eq. (\ref{eqn:oneC2}). Finally the properties of $p \in \cal G$
implies Eq. (\ref{eqn:oneC3}).
\end{proof}

Note that if the greatest element of $\{b,-b\}$ (denoted as $\psi$) exists, then $\psi$ is an upper bound of $\{b,-b\}$ and the condition on $a+c$ in Theorem \ref{thm:general_one} can be written as $a+c\succeq_I \psi$. In particular, if $b+b\preceq_I 0$, then $\psi = -b$.
If $b+b\succeq_I 0$, then $\psi=b$.
For ${\cal C}_2$, Lemma \ref{lem:gC} shows that polynomials with only even powers or only odd powers and coefficients in the positive cone are in the class $\cal G$.
Eq. (\ref{eqn:oneC2}) and Eq. (\ref{eqn:oneC3}) are true for nonassociative $\ast$ as long as the same order of operations is applied to the terms.
As an example of applying Theorem \ref{thm:general_one} to tuples in $\cal C$, consider the inner product $x\ast y = x^TAy$, where we get:
\begin{corollary} \label{cor:gc_matrix}
If $x,y,z\in \R^n$ are vectors such that $x\geq y\geq z$ elementwise and $a,b,c$ real numbers such that $a,c\geq 0$, $a+c\geq |b|$ and $A$ a nonnegative matrix, then
\begin{equation}
a(x-y)^TA(x-z) + b(y-x)^TA(y-z) + c(z-x)^TA(z-y) \geq 0
\end{equation}
\end{corollary}

Let $\preceq$ denote the Loewner order of matrices and $\otimes$ and $\circ$ denote the Kronecker product and the Hadamard product respectively. Then we have
\begin{corollary} \label{cor:gc_k_h}
If $A, B, C$ are Hermitian matrices of the same order such that $A\succeq B\succeq C$ and $a,b,c$ real numbers such that $a,c\geq 0$, $a+c\geq |b|$, then
\begin{eqnarray*}
& a(A-B)\otimes (A-C) + b(B-A)\otimes (B-C) + c(C-A)\otimes (C-B) \succeq 0\\
& a(A-C)\otimes (A-B) + b(B-C)\otimes (B-A) + c(C-B)\otimes (C-A) \succeq 0\\
& a(A-B)\circ (A-C) + b(B-A)\circ (B-C) + c(C-A)\circ(C-B) \succeq 0
\end{eqnarray*}
If in addition $A$, $B$ and $C$ commute, then 
$$a(A-B)(A-C) + b(B-A)(B-C) + c(C-A)(C-B) \succeq 0$$
\end{corollary}

Theorem \ref{thm:gc} can also be generalized. We use the notation $\Asterisk_i x_i$ to denote $x_1\ast x_2 \ast \cdots$ for a symmetric and associative operation $\ast$.

\begin{theorem} \label{thm:gcC}
Let $\ituple$ be a partially ordered commutative ring. 
Let $x_i\succeq_I x_{i+1}$ for $1\leq i \leq 3$ be such that $x_1+x_4\succeq_I x_2+x_3$ and $n\geq 0$ an integer.
If there exists an upper bound of $\{a_i, -a_i\}$ denoted as $\hat{a}_i$ for all even $i$, and $a_1 \succeq_I \hat{a}_2 \succeq_I a_3 \succeq_I \hat{a}_4 \succeq_I 0$, then
\begin{align}
\sum_{i=1}^4 a_i \ast \Asterisk_{j\neq i} (x_i-x_j)^n \succeq_I 0\label{eqn:gcC}
\end{align}
\end{theorem}

\begin{proof}
The left hand side of Eq. (\ref{eqn:gcC}), denoted as $\beta$, can be written as
\begin{align*} \beta = & a_1\ast (x_1-x_4)^n\ast (x_1-x_2)^n \ast (x_1-x_3)^n+a_4\ast (x_4-x_1)^n\ast (x_4-x_3)^n \ast (x_4-x_2)^n \\
+ & a_2\ast (x_2-x_3)^n\ast (x_2-x_1)^n\ast (x_2-x_4)^n + a_3\ast (x_3-x_2)^n\ast (x_3-x_1)^n\ast (x_3-x_4)^n
\end{align*}
Since $x_1-x_2\succeq_I x_3-x_4$ and $x_1-x_3\succeq_I  x_2-x_4$, we can bound this by:
\begin{align*} 
\beta \succeq_I &  (x_1-x_4)^n\ast(a_1\ast (x_1-x_2)^n\ast (x_1-x_3)^n-\hat{a}_4\ast (x_3-x_4)^n\ast (x_2-x_4)^n) \\
+ & (x_2-x_3)^n\ast(a_3\ast (x_1-x_3)^n\ast (x_3-x_4)^n -\hat{a}_2\ast (x_1-x_2)^n\ast (x_2-x_4)^n)\\
= & (x_1-x_4)^n\ast w_1+(x_2-x_3)^n\ast w_2
\end{align*}
where $w_1 = a_1\ast (x_1-x_2)^n\ast (x_1-x_3)^n-\hat{a}_4\ast (x_3-x_4)^n\ast (x_2-x_4)^n$ and $w_2 = a_3\ast (x_1-x_3)^n\ast (x_3-x_4)^n -\hat{a}_2\ast (x_1-x_2)^n\ast(x_2-x_4)^n$.
Thus $w_1\succeq_I 0$. By choosing $A=x_1-x_2$, $B=x_3-x_4$ and $C=x_2-x_3$, Lemma \ref{lem:gC}  shows that
$(x_1-x_3)^n\ast (x_3-x_4)^n  \preceq_I (x_1-x_2)^n\ast(x_2-x_4)^n$. Since $\hat{a}_2\succeq_I a_3$, this implies that
$w_2\preceq_I 0$. Since $(x_1-x_4)\succeq_I (x_2-x_3)$, this implies $(x_1-x_4)^n\ast w_2\preceq_I (x_2-x_3)^n\ast w_2$, which means that
 $\beta \succeq_I (x_1-x_4)^n\ast (w_1+w_2)$.  Next $w_1+w_2$ can be written as:
\begin{align*}
w_1+w_2 =&  (x_1-x_2)^n\ast(a_1\ast(x_1-x_3)^n-\hat{a}_2\ast(x_2-x_4)^n)  \\
+ &(x_3-x_4)^n\ast(a_3\ast(x_1-x_3)^n-\hat{a}_4\ast(x_2-x_4)^n)
\end{align*}
which is $\succeq_I 0$ since $a_1\succeq_I \hat{a}_2\succeq_I a_3\succeq_I \hat{a}_4$ and $x_1-x_3\succeq_I x_2-x_4$.
\end{proof}

Corollary \ref{cor:gc5} can be generalized as:

\begin{theorem} \label{thm:gcC5}
Let $\ituple$ be a partially ordered commutative ring. Let $x_i\succeq_I x_{i+1}$ for $1\leq i \leq 4$ be such that $x_1+x_4\succeq_I x_2+x_3$ and $n\geq 0$ an integer. Assume there exists an upper bound of $\{a_i, -a_i\}$ denoted as $\hat{a}_i$ for all even $i$ and $a_1 \succeq_I \hat{a}_2 \succeq_I a_3 \succeq_I 0$.
 If $a_3+a_5\succeq_I \hat{a}_4 \succeq_I a_5 \succeq_I 0$ or $a_5 \succeq_I \hat{a}_4\succeq_I 0$,  then
\begin{align}
\sum_{i=1}^5 a_i \ast \Asterisk_{j\neq i} (x_i-x_j)^n \succeq_I 0 \label{eqn:gcC5}
\end{align}
\end{theorem}

\begin{proof}
Without loss of generality, we assume that $\hat{a}_i = a_i$ as the general case follows readily\footnote{We will make this assumption in the subsequent results as well.} and relies on the fact that for even $i$, $\hat{a}_i\ast\Asterisk_{j<i}(x_j-x_i)^n\ast\Asterisk_{j>i}(x_i-x_j)^n\succeq_I a_i\ast\Asterisk_{j\neq i}(x_i-x_j)^n$. If $a_5 \succeq_I a_4$, then $a_5\ast\Asterisk_{j<5}(x_5-x_j)^n$ can be written as $a_4\ast\Asterisk_{j<5}(x_5-x_j)^n + (a_5-a_4)\ast\Asterisk_{j<5}(x_5-x_j)^n$. Since $(a_5-a_4)\ast\Asterisk_{j<5}(x_5-x_j)^n \succeq_0$, this is reduced to the case $a_5=a_4$ which is a special case of we will consider next.
Next consider the case where $a_3\succeq_I a_4-a_5\succeq_I 0$. Let use define the variables $\tilde{a}_i = a_i\ast (x_i-x_5)^n$ for $1\leq i\leq 4$ and $\tilde{a}_5 = a_5\ast (x_4-x_5)^n$. 
Then the hypothesis implies that
$\tilde{a}_1 \succeq_I \tilde{a}_2 \succeq_I \tilde{a}_3 \succeq_I \tilde{a}_4-\tilde{a}_5\succeq_I 0$. 
The left hand side of Eq. (\ref{eqn:gcC5}) can be rewritten as
\begin{align}&\sum_{i=1}^3\tilde{a}_i\ast \Asterisk_{i\neq j,j<5}(x_i-x_j)^n   + (\tilde{a}_4-\tilde{a}_5)\ast\Asterisk_{j=1}^3(x_4-x_j)^n\nonumber\\&+a_5\ast\left(\Asterisk_{j\neq 5}(x_5-x_j)^n+\Asterisk_{j\neq 4}(x_4-x_j)^n\right)\label{eqn:lhsgc5}
\end{align}
Since  $a_5\succeq_I 0$ and $\Asterisk_{j\neq 5}(x_5-x_j)^n\succeq_I -\Asterisk_{j\neq 4}(x_4-x_j)^n$, the last term in Eq. (\ref{eqn:lhsgc5}) $\succeq_I  0$.
By applying Theorem \ref{thm:gcC} to $\sum_{i=1}^3\tilde{a}_i\ast\Asterisk_{i\neq j,j<5}(x_i-x_j)^n   + (\tilde{a}_4-\tilde{a}_5)\ast\Asterisk_{j=1}^3(x_4-x_j)^n$ we see that Eq. (\ref{eqn:gcC5}) holds and the proof is complete.
\end{proof}

It is straightforward to extend the results in Section \ref{sec:67} as well. 
Recall that Lemma \ref{lem:gC} and Lemma \ref{lem:lemma3G} imply that the power functions $x\rightarrow x^n$ for integers $n\geq 0$ are in class ${\cal G}_2$ 
and we have the following analogous results to Theorem \ref{thm:gc6} and Corollary \ref{cor:gc7}.

\begin{theorem} \label{thm:gcC6}
Let $\ituple$ be a partially ordered commutative ring.
 Let $x_i\succeq_I x_{i+1}$ for $1\leq i \leq 5$ be such that $x_1+x_6\succeq_I x_2+x_5\succeq_I x_3+x_4$ and $n\geq 0$ an integer. Suppose $(x_2-x_6)^{-1}\succ_I 0$ exists.
 If there exists an upper bound of $\{a_i, -a_i\}$ denoted as $\hat{a}_i$ for all even $i$, $a_1 \succeq_I \hat{a}_2 \succeq_I a_5 \succeq_I \hat{a}_6\succeq_I 0$ and $a_3 \succeq_I \hat{a}_4 \succeq_I 0$ then
\begin{align}
\sum_{i=1}^6 a_i \ast \Asterisk_{j\neq i} (x_i-x_j)^n \succeq_I 0 \label{eqn:gcC6}
\end{align}
\end{theorem}

\begin{proof}
Define  $g_{i,j} = (x_i-x_j)^n$ and $\gamma_i = \Asterisk_{j\neq i}g_{i,j}$. Note that $\gamma_i \succeq_I 0$ if the index $i$ is odd and $\gamma_i \preceq_I 0$ otherwise.  Now the left hand side of Eq. (\ref{eqn:gcC6}) can be written as $\sum_{i=1}^6 a_i\ast\gamma_i$. For simplicity, we write $x\ast y$ as $xy$.
Since $a_6 \gamma_6 = a_5\gamma_6 + (a_6-a_5)\gamma_6$ and $(a_6-a_5)\gamma_6 \succeq_I 0$, without loss of generalization we can assume that $a_5 = a_6$.
It is easy to see that $\gamma_1 +\gamma_2\succeq_I 0$. Consider the 2 terms $\gamma_3$ and $\gamma_4$.
Since $x_1-x_3\succeq_I x_4-x_6$, by setting
$A = x_1-x_3$, $B = x_4-x_6$ and $C = x_3-x_4$ in Lemma \ref{lem:gC}, we see that $g_{1,3}g_{3,6} \succeq_I g_{1,4}g_{4,6}$. 
Similarly by setting
$A = x_2-x_3$, $B = x_4-x_5$ and $C = x_3-x_4$  we see that $x_2-x_3\succeq_I x_4-x_5$ implies $g_{2,3}g_{3,5}\succeq_I g_{2,4}g_{4,5}$. This means that $a_3\gamma_3+a_4\gamma_4 \succeq_I a_3(\gamma_3+\gamma_4) \succeq_I 0$. 

Since $x_1-x_2 \succeq_I x_5-x_6$ by setting $A = x_1-x_2$, $B = x_5-x_6$ and $C = x_2-x_5$ we see that $g_{1,2}g_{2,6}\succeq_I g_{1,5}g_{5,6}$. $x_2-x_4\succeq_I x_3-x_5$ implies that $g_{2,4}\succeq_I g_{3,5}$ and  $x_2-x_3\succeq_I x_4-x_5$ implies that $g_{2,3}\succeq_I g_{4,5}$, i.e. $g_{2,3}g_{2,4} \succeq_I g_{3,5}g_{4,5}$. This shows that $\gamma_2+\gamma_5 \preceq_I 0$.

Next we show that $a_1\gamma_1+a_5\gamma_5 \succeq_I -a_2\gamma_2-a_6\gamma_6$. Let us define $\eta =g_{1,5}g_{1,6}g_{2,6}\succeq_I 0$.
Since $x_2-x_6 \preceq_I x_1-x_5$, this implies that $g_{2,6}\preceq_I g_{1,5}$ and
$g_{2,6}a_5\gamma_5 +g_{1,5}a_2\gamma_2 \preceq_I g_{2,6}(a_5\gamma_5+a_2\gamma_2) \preceq_I g_{2,6}(a_5(\gamma_5+\gamma_2))\preceq_I 0$. 
Since $g_{1,6}\succeq_I g_{2,5}$,
it is straightforward to show that
\begin{align}
0 &\succeq_I g_{2,6}(a_5\gamma_5+a_2\gamma_2) \succeq_I g_{2,6}a_5\gamma_5 +g_{1,5}a_2\gamma_2  \nonumber \\
& = g_{1,5}g_{2,5}g_{2,6}\left(a_5g_{3,5}g_{4,5}g_{5,6}-a_2g_{1,2}g_{2,3}g_{2,4}\right) \succeq_I \eta\left(a_5g_{3,5}g_{4,5}g_{5,6}-a_2g_{1,2}g_{2,3}g_{2,4}\right) 
\label{eqn:gc6-eq1}
\end{align}
Similarly
\begin{align}
&g_{2,6}(a_1\gamma_1+a_6\gamma_6) \succeq_I g_{2,6}a_1\gamma_1+g_{1,5}a_6\gamma_6 \nonumber \\  
& = g_{1,5}g_{1,6}g_{2,6} \left(a_1g_{1,2}g_{1,3}g_{1,4} -a_6g_{3,6}g_{4,6}g_{5,6}\right) =
\eta\left(a_1g_{1,2}g_{1,3}g_{1,4} -a_6g_{3,6}g_{4,6}g_{5,6} \right)
\label{eqn:gc6-eq2}
\end{align}
We add Eq. (\ref{eqn:gc6-eq1}) to Eq. (\ref{eqn:gc6-eq2}) to obtain
\begin{align*} &g_{2,6}\left(a_1\gamma_1+a_5\gamma_5+a_2\gamma_2+a_6\gamma_6\right)
\\ & \succeq_I \eta(a_5g_{3,5}g_{4,5}g_{5,6}-a_2g_{1,2}g_{2,3}g_{2,4}+a_1g_{1,2}g_{1,3}g_{1,4}-a_6g_{3,6}g_{4,6}g_{5,6})
\end{align*}
The right hand side of the equation above can be written as $\eta\delta$ where
\begin{equation}\label{eqn:term1}
\delta  = g_{5,6}\left(a_5g_{3,5}g_{4,5}-a_6g_{3,6}g_{4,6}\right) + g_{1,2}\left(a_1g_{1,3}g_{1,4}-a_2g_{2,3}g_{2,4}\right)
\end{equation}
  Since $a_3\gamma_3+a_4\gamma_4\succeq_I 0$, the conclusion follows if $\delta \succeq_I 0$. Let $\mu = a_5g_{3,5}g_{4,5}-a_6g_{3,6}g_{4,6}$. Note that $g_{1,2}\left(a_1g_{1,3}g_{1,4}-a_2g_{2,3}g_{2,4}\right) \succeq_I 0$. 
Since $a_5=a_6$, this implies that $\mu = a_5(g_{3,5}g_{4,5}-g_{3,6}g_{4,6})\preceq_I 0$. Since $g_{1,2}\succeq_I g_{5,6}$ and $a_1\succeq_I a_5$,
\begin{equation*}
0 \succeq_I g_{5,6}\mu \succeq_I g_{1,2}\mu = a_5g_{1,2}\left(g_{3,5}g_{4,5}-g_{3,6}g_{4,6}\right) \succeq_I a_1g_{1,2}\left(g_{3,5}g_{4,5}-g_{3,6}g_{4,6}\right)
\end{equation*}
Similarly $a_1g_{1,4}g_{1,3}-a_2g_{2,3}g_{2,4}
\succeq_I a_1(g_{1,4}g_{1,3}-g_{2,3}g_{2,4})$. This means that
\begin{align*}
\delta & \succeq_I a_1g_{1,2}(g_{1,3}g_{1,4}+g_{3,5}g_{4,5}-g_{2,3}g_{2,4}-g_{3,6}g_{4,6})
\end{align*}
Lemma \ref{lem:lemma3G} implies that $g_{1,3}g_{1,4}+g_{3,5}g_{4,5}-g_{2,3}g_{2,4}-g_{3,6}g_{4,6}\succeq_I 0$ by choosing $x=x_2-x_3$, $z=x_2-x_4$, $w=x_3-x_5$, $y=x_4-x_5$, $t = x_1-x_2$ and $u = x_5-x_6$ and this implies that $\delta \succeq_I 0$. Therefore $\eta\delta \succeq_I 0$ and thus $g_{2,6}\left(a_1\gamma_1+a_5\gamma_5+a_2\gamma_2+a_6\gamma_6\right)\succeq_I 0$. Since $g_{2,6}^{-1}\succ_I 0$ by hypothesis, this implies $a_1\gamma_1+a_5\gamma_5+a_2\gamma_2+a_6\gamma_6 \succeq_I 0$. Since $a_3\gamma_3+a_4\gamma_4\succeq_I 0$, the proof is complete.
\end{proof}

Note that for the Loewner order $\succ$ and $\ast$ being matrix multiplication, if $x_2\succ x_6$, then $x_2-x_6\succ 0$ is a positive definite matrix whose inverse $(x_2-x_6)^{-1} \succ 0$ is also positive definite.

\begin{theorem} \label{thm:gcC7}
Let $\ituple$ be a partially ordered commutative ring.
 Let $x_i\succeq_I x_{i+1}$ for $1\leq i \leq 6$ be such that $x_1+x_6\succeq_I x_2+x_5\succeq_I x_3+x_4$ and $n\geq 0$ an integer. Suppose $(x_2-x_6)^{-1}\succ_I 0$ exists.
 If there exists an upper bound of $\{a_i, -a_i\}$ denoted as $\hat{a}_i$ for all even $i$, $a_1 \succeq_I \hat{a}_2 \succeq_I  a_5 \succeq_I \hat{a}_6\succeq_I 0$, $a_3\succeq_I \hat{a}_4 \succeq_I 0$, and either $a_5 + a_7 \succeq_I \hat{a}_6\succeq_I a_7\succeq_I 0$ or $a_7\succeq_I \hat{a}_6\succeq_I 0$ then
\begin{align}
\sum_{i=1}^7 a_i \ast \Asterisk_{j\neq i} (x_i-x_j)^n   \succeq_I 0 \label{eqn:gcC7}
\end{align}
\end{theorem}

\begin{proof}
We define $g_{ij}$ and $\gamma_i$ as in the proof of Theorem \ref{thm:gcC6}.
If $a_7 \succeq_I a_6$, then $a_7\ast \gamma_7$ can be written as $a_6\ast \gamma_7 + (a_7-a_6)\ast \gamma_7$. Since $(a_7-a_6)\ast\gamma_7 \succeq_0$, this is reduced to the case $a_7=a_6$ which is a special case of we will consider next.

Next suppose that $a_5 + a_7 \succeq_I a_6 \succeq_I a_7$.
Let us define the variables $\tilde{a}_i = a_i\ast g_{i,7}$ for $1\leq i\leq 6$ and $\tilde{a}_7 = a_7\ast g_{6,7}$. Then the hypothesis implies that
$\tilde{a}_1 \succeq_I \tilde{a}_2 \succeq_I \tilde{a}_5\succeq_I \tilde{a}_6-\tilde{a}_7\succeq_I 0$, $\tilde{a}_3\succeq_I \tilde{a}_4$. 
The left hand side of Eq. (\ref{eqn:gcC7}) can be rewritten as
\begin{equation}\sum_{i=1}^5\tilde{a}_i\ast \Asterisk_{i\neq j,j<7}g_{i,j}   + (\tilde{a}_6-\tilde{a}_7)\ast \Asterisk_{j=1}^5g_{6,j}
+a_7\ast\left(\gamma_7+\gamma_6\right)
\label{eqn:lhsgc7C}
\end{equation}
Since  $a_7\succeq_I 0$ and $\gamma_7+\gamma_6\succeq_I 0$,
the last term in Eq. (\ref{eqn:lhsgc7C}) is $\succeq_I 0$.
By applying Theorem \ref{thm:gcC6} to $\sum_{i=1}^5\tilde{a}_i\ast\Asterisk_{i\neq j,j<7}g_{i,j}   + (\tilde{a}_6-\tilde{a}_7)\ast\Asterisk_{j=1}^5g_{6,j}$ the proof is complete.
\end{proof}

\section{Conclusions}
Schur's inequality gives conditions under which the sum of products of nontrivial differences among 3 real numbers is nonnegative. We proved several generalizations of Schur's inequality that include multiple variables, more general functions of the differences and products of vectors and matrices. We also show that a Schur's inequality of $2n$ variables leads to a Schur's inequality of $2n+1$ variables.

\end{document}